\documentclass[reqno]{amsart}
\title[Proving some identities of Gosper...]{Proving some identities of Gosper on $q$-trigonometric functions}
\usepackage{amssymb,amsmath,amsthm,epsfig,graphics,latexsym}
\usepackage{enumerate}
 \theoremstyle{definition}

  \theoremstyle{plain}

  \newtheorem{theorem}    {Theorem}
  
  \newtheorem{corollary}  {Corollary}

  \newcommand{\te}{\theta}
  \newcommand{\fr}{\frac}
  \newcommand{\lt}{\left(}
  \newcommand{\rt}{\right)}
\begin{document}
  \author{Mohamed El Bachraoui}
 \email{melbachraoui@uaeu.ac.ae}
  \begin{abstract}
  Gosper introduced the functions $\sin_q z$ and $\cos_q z$ as $q$-analogues for the trigonometric functions $\sin z$ and $\cos z$ respectively.
  He stated but did not prove a variety of identities involving these two $q$-trigonometric functions.
  In this paper,
  we shall use the theory of elliptic functions to prove three formulas from the list of Gosper on the functions $\sin_q z$ and $\cos_q z$.
  \end{abstract}
  \date{\textit{\today}}
  \maketitle
\section{Introduction}
\noindent
Throughout the paper let $q=e^{\pi i\tau}$ with $\mathrm{Im}(\tau)>0$, let $\tau'=-\frac{1}{\tau}$, and let $p=e^{\pi i\tau'}$. Note that the assumption $\mathrm{Im}(\tau)>0$ guarantees
that $|q|<1$ and $|p|<1$.
For a complex variable $a$, the $q$-shifted factorials are given by
\[
(a;q)_0= 1,\quad (a;q)_n = \prod_{i=0}^{n-1}(1-a q^i),\quad
(a;q)_{\infty} = \lim_{n\to\infty}(a;q)_n
\]
and for brevity let
\[
(a_1,\ldots,a_k;q)_n = (a_1;q)_n\cdots (a_k;q)_n,\quad
(a_1,\ldots,a_k;q)_{\infty} = (a_1;q)_{\infty} \cdots (a_k;q)_{\infty}.
\]
The first Jacobi theta function is defined
as follows:
\[
\theta_1(z,q)=\theta_1(z \mid \tau) = 2\sum_{n=0}^{\infty}(-1)^n q^{(2n+1)^2/4}\sin(2n+1)z,
\]
which has the following very useful infinite product representations
\[
\theta_1(z \mid \tau)=
i q^{\frac{1}{4}}e^{-iz} (q^2 e^{-2iz},e^{2iz},q^2; q^2)_{\infty} = 2 q^{\frac{1}{4}} (\sin z) (q^2 e^{2iz}, q^2 e^{-2iz},q^2; q^2)_{\infty}.
\]
See Whittaker~and~Watson~\cite[p. 469]{Whittaker-Watson} and
Gasper~and~Rahman~\cite[p. 15]{Gasper-Rahman}.
For the purpose of this work we will need the following basic properties of the function
$\theta_1(z \mid \tau)$
which can be derived straightforwardly by the definitions.
\begin{align}\label{theta-1}
\theta_1(k\pi\mid\tau) &= 0 \quad (k\in\mathbb{Z}), \nonumber \\
\theta_1 \lt \fr{\pi}{2}\bigm| \tau \rt &= 2 q^{\fr{1}{4}}(-q^2;q^2)_{\infty}^2 (q^2;q^2)_{\infty}, \nonumber \\
 \theta_1(z+\pi \mid \tau) &= -\theta_1(z\mid \tau) = \te_1(-z\mid \tau),  \\
\theta_1(z+\pi\tau \mid \tau) &= -q^{-1} e^{-2iz} \theta_1(z \mid \tau), \nonumber \\
\te_1'(0\mid\tau) &= 2 q^{\fr{1}{4}}(q^2;q^2)^3. \nonumber
\end{align}
Moreover, one can show that for any positive integer $k$ we have
 \begin{equation} \label{Transf}
 \te_1 \left( z+\pi\tau \bigm| \frac{\tau}{k} \right) = (-1)^k q^{-k}e^{-2kiz}
 \te_1 \left(z\bigm| \frac{\tau}{k} \right).
 \end{equation}
Further, Jacobi's imaginary transformation for the function $\theta_1(z \mid \tau)$ states that
\begin{equation}\label{ImTransf}
\theta_1(z\mid \tau) = (-i\tau)^{-\frac{1}{2}} (-i) e^{\frac{i\tau' z^2}{\pi}}\theta_1(z\tau' \mid \tau').
\end{equation}
See \cite[p. 475]{Whittaker-Watson}.
Differentiating (\ref{ImTransf}) with respect to $z$ and combining with the basic properties
(\ref{theta-1}) we find
\begin{equation} \label{key-derivative}
\te_1'(0\mid \tau') = (-i\tau)^{\fr{3}{2}} 2 q^{\fr{1}{4}} (q^2;q^2)^3.
\end{equation}
Gosper~\cite{Gosper} introduced $q$-analogues of $\sin(z)$ and $\cos(z)$ as follows:
\[
\sin_q (\pi z) = q^{(z-1/2)^2}\prod_{n=1}^{\infty}
\frac{(1-q^{2n-2z})(1-q^{2n+2z-2})}{(1-q^{2n-1})^2} =
q^{(z-\frac{1}{2})^2} \frac{(q^{2z},q^{2-2z};q^2)_{\infty}}{(q;q^2)_{\infty}^2},
\]
\[
\cos_q (\pi z) = q^{z^2}\prod_{n=1}^{\infty}
\frac{(1-q^{2n-2z-1})(1-q^{2n+2z-1})}{(1-q^{2n-1})^2}
= q^{z^2} \frac{(q^{1+2z},q^{1-2z};q^2)_{\infty}}{(q;q^2)_{\infty}^2}.
\]
It is easy to see that $\cos_q (z) = \sin_q (\pi/2 - z)$.
Gosper proved a variety of identities involving  these two functions including
the following
\begin{equation*}
\sin_q (z) = \frac{\theta_1(z,p)}{\theta_1\left( \frac{\pi}{2},p \right)}\quad \text{where\ }
(\ln p) (\ln q) = \pi^2,
\end{equation*}
which is readily seen to be equivalent to
\begin{equation}\label{sine-theta}
\sin_q (z) = \frac{\theta_1(z\mid \tau')}{\theta_1\left( \frac{\pi}{2}\bigm| \tau' \right)}.
\end{equation}
Clearly the formula (\ref{sine-theta}) combined with the
identities $\cos_q (z) = \sin_q (\pi/2 - z)$ and $\theta_1(z+\pi)=-\theta_1(z)$  yield
\begin{equation}\label{cosine-theta}
\cos_q (z) = \frac{\theta_1\left( z+\frac{\pi}{2},p \right)}{\theta_1 \left( \frac{\pi}{2},p\right)}
 = \frac{\theta_1\left( z+\frac{\pi}{2} \bigm| \tau' \right)}{\theta_1 \left( \frac{\pi}{2} \bigm| \tau' \right)}.
\end{equation}
See Gosper~\cite[p. 98]{Gosper}.
The author also stated without proof many identities based on a computer algebra facility called MACSYMA. Among the
formulas, we find the following which we will mark with the same labels as in
 Gosper~\cite{Gosper}.
\[   \label{q-Double} \tag{$q$-Double}
\sin_q(2z) = \frac{1}{2}\frac{\Pi_q}{\Pi_{q^4}} \sqrt{(\sin_{q^4}z)^2- (\sin_{q^2}z)^4 },
\]
\[
\label{q-Double-3} \tag{$q$-Double$_3$}
\cos_q(2z) = (\cos_{q^2} z)^2 - (\sin_{q^2} z)^2,
\]
\[ \label{q-Triple} \tag{$q$-Triple}
\sin_q(3z) = \frac{1}{3}\frac{\Pi_q}{\Pi_{q^9}} \sin_{q^9}z - \left(1+\frac{1}{3}\frac{\Pi_q}{\Pi_{q^9}}\right) (\sin_{q^3}z)^3,
\]
and
\[ \label{q-Triple-2} \tag{$q$-Triple$_2$}
\sin_q(3z) = \frac{\Pi_q}{\Pi_{q^3}} (\cos_{q^3} z)^2 \sin_{q^3}z -  (\sin_{q^3}z)^3,
\]
where
\begin{equation}\label{Piq}
\Pi_q = q^{\frac{1}{4}}\frac{ (q^2;q^2)_{\infty}^2}{(q;q^2)_{\infty}^2}.
\end{equation}
As it was observed by Gosper, it is easy to verify that (\ref{q-Double}) is equivalent to
\[   \label{q-Double-4} \tag{$q$-Double$_4$}
\cos_q(2z) = \frac{1}{2}\frac{\Pi_q}{\Pi_{q^4}} \sqrt{(\cos_{q^4}z)^2- (\cos_{q^2}z)^4 }
\]
and that by a combination of (\ref{q-Double}), (\ref{q-Double-3}, and (\ref{q-Double-4}) we have
\[   \label{q-Double-5} \tag{$q$-Double$_5$}
\cos_q(2z) = (\cos_{q}z)^4- (\sin_{q}z)^4.
\]
See Gosper~\cite[p. 89--93]{Gosper}.
We notice that many of Gosper's formulas are $q$-analogues for well-known trigonometric functions. For instance, after introducing the function
$\cos_q z$ as in (\ref{cosine-theta}), Gosper showed that
\[ \label{q-Double-2} \tag{$q$-Double$_2$}
\sin_q(2z) = \frac{\Pi_q}{\Pi_{q^2}} \sin_{q^2}z \cos_{q^2}z,
\]
which is clearly a $q$-analogue of the famous formula $\sin 2z = 2 \cos z \sin z$.
Recently, Mez\H{o}~\cite{Mezo} gave a different proof for (\ref{q-Double-2}).
In his proof, Mez\H{o} among other things analysed the logarithmic derivatives with respect to $z$ of
$\sin_q(2z)$ and
$\sin_{q^2} (z) \cos_{q^2}(z)$ and found that they coincide, implying that the ratio
\begin{equation}\label{ratio}
\frac{\sin_q(2z)}{\sin_{q^2} (z) \cos_{q^2}(z)} = C(q),
\end{equation}
for a constant $C(q)$ which he was able to determine
as in the formula (\ref{q-Double-2}).
Alternatively, taking into account the relations (\ref{sine-theta}) and (\ref{cosine-theta}),
formula~(\ref{q-Double-3}) can be
written as
\[
\frac{\theta_1\left(2z+\frac{\pi}{2} \bigm| \tau' \right)}{\theta_1\left(\frac{\pi}{2}\bigm| \tau' \right)} =
\left( \frac{\theta_1\left(z+\frac{\pi}{2} \bigm| \frac{\tau'}{2} \right)}{\theta_1\left(\frac{\pi}{2}\bigm| \frac{\tau'}{2}\right)}\right)^2
-
\left( \frac{\theta_1\left(z \bigm| \frac{\tau'}{2} \right)}{\theta_1\left(\frac{\pi}{2} \bigm| \frac{\tau'}{2}\right)} \right)^2,
\]
which after rearrangement becomes
\begin{equation}\label{help-0}
\theta_1 \left(2z+ \frac{\pi}{2} \bigm|  \tau'\right) \theta_1 ^2 \left( \frac{\pi}{2} \bigm|  \frac{\tau'}{2} \right)
=
\theta_1 \left(\frac{\pi}{2} \bigm|  \tau'\right) \theta_1 ^2 \left( z+\frac{\pi}{2} \bigm| \frac{\tau'}{2} \right)
- \theta_1 \left(\frac{\pi}{2} \bigm| \tau'\right) \theta_1 ^2 \left( z \bigm|  \frac{\tau'}{2} \right).
\end{equation}
Furthermore, again by virtue of (\ref{sine-theta}) and (\ref{cosine-theta}) note that formula
(\ref{q-Double-2}) means
\[
\frac{\theta_1(2z \mid \tau')}{\theta_1\left(\frac{\pi}{2} \bigm| \tau' \right)} =
C(q)\frac{\theta_1\left(z\bigm| \frac{\tau'}{2} \right)}{\theta_1\left(\frac{\pi}{2} \bigm| \frac{\tau'}{2} \right)} \cdot
\frac{\theta_1\left(z+\frac{\pi}{2} \bigm| \frac{\tau'}{2} \right)}
{\theta_1\left(\frac{\pi}{2} \bigm| \frac{\tau'}{2} \right)},
\]
or equivalently,
\[
\theta_1(2z \mid \tau')\theta_1^2\left(\frac{\pi}{2}\bigm| \frac{\tau'}{2} \right) =
C(q) \theta_1\left(\frac{\pi}{2}\bigm| \tau' \right)
\theta_1\left(z\bigm| \frac{\tau'}{2} \right) \theta_1\left(z+\frac{\pi}{2}\bigm| \frac{\tau'}{2} \right),
\]
which by observing that
\[
C(q) = \fr{ \te_1^2\lt \fr{\pi}{2}\bigm| \fr{\tau'}{2} \rt}{\te_1^2\lt \fr{\pi}{4}\bigm| \fr{\tau'}{2} \rt}
\]
means that
\begin{equation} \label{help}
\te_1 \lt z\bigm| \fr{\tau'}{2}\rt \te_1 \lt z+\fr{\pi}{2}\bigm| \fr{\tau'}{2}\rt
\te_1 \lt \fr{\pi}{2}\bigm| \tau'\rt =  \te_1 \lt 2z\bigm| \tau' \rt \te_1^2 \lt \fr{\pi}{4}\bigm| \fr{\tau'}{2}\rt.
\end{equation}
After recognising the equivalent forms (\ref{help-0}) and (\ref{help}) as three-term addition
formulas involving theta functions,
 this author in~\cite{Bachraoui} proved both
(\ref{q-Double-2}) and (\ref{q-Double-3}) by employing the theory of elliptic functions.
Moreover, by following the same steps as before, we can check that each one of Gosper's identities
(\ref{q-Double}), (\ref{q-Triple}), and (\ref{q-Triple-2}) is
a three-term addition formulas involving theta functions. See below Theorem~\ref{main-result-1},
Theorem~\ref{main-result-2}, and Theorem~\ref{main-result-3}.
The theory of elliptic functions proved to be a powerful tool to study this type of addition formulas.
For recent papers dealing with addition formulas by means of elliptic functions,
we refer to Liu~\cite{Liu-2005, Liu-2007}. See also Whittaker~and~Watson~\cite{Whittaker-Watson}
and Shen~\cite{Shen-1, Shen-2}
for more additive formulas involving theta functions and applications. For a more direct approach
to produce addition formulas involving theta functions through
infinite sums manipulations, we refer to the book by Lawden~\cite{Lawden}.
In this paper our goal is to
prove Gosper's formulas (\ref{q-Double}), (\ref{q-Triple}),
and (\ref{q-Triple-2}) using the theory of elliptic functions.
The paper is organized as follows. In Section~\ref{sec-main-results} we state three theorems and prove
three corollaries which are the equivalent forms of
(\ref{q-Double}), (\ref{q-Triple}), and (\ref{q-Triple-2}). In Section~\ref{sec-general-result} we state and
prove a general result which we shall need
to prove the three theorems in Section~\ref{sec-main-results}. Section~\ref{sec-proof-main-result-1},
Section~\ref{sec-proof-main-result-2}, and Section~\ref{sec-proof-main-result-3},
are devoted to proofs for Theorem~\ref{main-result-1}, Theorem~\ref{main-result-2}, and
Theorem~\ref{main-result-3} respectively.
%
\section{Main results}\label{sec-main-results}
\begin{theorem}\label{main-result-1}
For all complex number $z$ we have
 \[
 \begin{split}
& \theta_1^2(2z\mid\tau') \theta_1^4\left( \frac{\pi}{2}\bigm| \frac{\tau'}{2}\right)
\theta_1^2\left(\frac{\pi}{2}\bigm| \frac{\tau'}{4}\right)   \\
= &
\left( \frac{1}{2}\frac{\Pi_q}{\Pi_{q^4}}\right)^2 \theta_1^2\left(z \bigm| \frac{\tau'}{4}\right)
\theta_1^4\left(\frac{\pi}{2} \bigm| \frac{\tau'}{2}\right)
 \theta_1^2\left(\frac{\pi}{2}\mid\tau' \right)   \\
- &
 \left( \frac{1}{2}\frac{\Pi_q}{\Pi_{q^4}}\right)^2 \theta_1^4\left(z \bigm| \frac{\tau'}{2}\right)
 \theta_1^2\left(\frac{\pi}{2} \bigm| \frac{\tau'}{4}\right) \theta_1^2\left(\frac{\pi}{2}\mid\tau' \right).
 \end{split}
 \]
\end{theorem}
\begin{corollary}\label{main-cor-1}
For all complex number $z$ we have
\[
\sin_q(2z) = \frac{1}{2}\frac{\Pi_q}{\Pi_{q^4}} \sqrt{(\sin_{q^4}z)^2- (\sin_{q^2}z)^4 }.
\]
\end{corollary}
\begin{proof}
By the relations (\ref{sine-theta}) and (\ref{cosine-theta}), we can readily
 see that (\ref{q-Double}), which is the statement of this corollary, is an equivalent form of Theorem~\ref{main-result-1}.
\end{proof}
\begin{theorem}\label{main-result-2}
For all complex number $z$ we have
\begin{align*}
 \theta_1(3z\mid\tau') \theta_1^3\left( \frac{\pi}{2}\bigm| \frac{\tau'}{3}\right)
\theta_1\left(\frac{\pi}{2}\bigm| \frac{\tau'}{9}\right)
 =
\frac{1}{3}\frac{\Pi_q}{\Pi_{q^9}} \theta_1\left(z \bigm| \frac{\tau'}{9}\right) \theta_1\left(\frac{\pi}{2} \bigm| \tau'\right)
 \theta_1^3\left(\frac{\pi}{2}\mid \frac{\tau'}{3} \right) \\
 -
 \left( \frac{1}{3}\frac{\Pi_q}{\Pi_{q^9}}+1\right) \theta_1^3\left(z \bigm| \frac{\tau'}{3}\right)
 \theta_1\left(\frac{\pi}{2} \bigm| \tau'\right) \theta_1\left(\frac{\pi}{2}\bigm| \frac{\tau'}{9} \right) .
\end{align*}
\end{theorem}
\begin{corollary}\label{main-cor-2}
For all complex number $z$ we have
\[
\sin_q(3z) = \frac{1}{3}\frac{\Pi_q}{\Pi_{q^9}} \sin_{q^9}z - \left(1+\frac{1}{3}\frac{\Pi_q}{\Pi_{q^9}}\right) (\sin_{q^3}z)^3.
\]
\end{corollary}
\begin{proof}
By the relations (\ref{sine-theta}) and (\ref{cosine-theta}), it is easy to check that the identity (\ref{q-Triple}) is equivalent to Theorem~\ref{main-result-2}.
\end{proof}
\begin{theorem}\label{main-result-3}
For all complex number $z$ we have
\begin{align*}
 \theta_1(3z\mid\tau') \theta_1^3\left( \frac{\pi}{2}\bigm| \frac{\tau'}{3}\right)
 +  \theta_1^3\left(z \bigm| \frac{\tau'}{3}\right)  \theta_1\left(\frac{\pi}{2} \bigm| \tau'\right) \\
 =
\frac{\Pi_q}{\Pi_{q^3}}
\theta_1\left(z \bigm| \frac{\tau'}{3}\right) \te_1^2\left(z+\frac{\pi}{2}\bigm| \frac{\tau'}{3}\right)
 \te_1\left( \fr{\pi}{2}\bigm| \fr{\tau'}{3}\right).
 \end{align*}
\end{theorem}
\begin{corollary}\label{main-cor-3}
For all complex number $z$ we have
\[
\sin_q(3z) = \frac{\Pi_q}{\Pi_{q^3}} (\cos_{q^3} z)^2 \sin_{q^3}z -  (\sin_{q^3}z)^3.
\]
\end{corollary}
\begin{proof}
By the relations (\ref{sine-theta}) and (\ref{cosine-theta}), it is easy to check that the identity (\ref{q-Triple-2}) is equivalent to Theorem~\ref{main-result-3}.
\end{proof}
\section{A general result}\label{sec-general-result}
\noindent
We note that Theorem~\ref{main-result-1}, Theorem~\ref{main-result-2}, and
Theorem~\ref{main-result-3} below are consequences of the following result which is an extension of a theorem of Liu~\cite[Theorem 1]{Liu-2007}.
\begin{theorem}\label{main-theorem}
Let $k$ be a positive integer, let $l$ be a nonnegative integer, and let $h_1(z)$ and $h_2(z)$ be entire functions satisfying the following
two conditions:
\[
h(z+\pi) = (-1)^l h(z)\ \text{and\ }
h(z+\fr{\pi\tau}{k}) = (-1)^{l} q^{\fr{-(2+l)}{k}} e^{-2(2+l)iz} h(z).
\]
If $h_1(z)$ and $h_2(z)$ have a zero of order al least $l$ at $z=0$, then there is a constant $C$ such
 that
\begin{align} \label{master}
 \left( h_1(x)+ (-1)^l h_1(-x) \right)  \left( h_2(y)+ (-1)^l h_2(-y) \right) \nonumber \\
 -
 \left( h_2(x)+ (-1)^l h_2(-x) \right)  \left( h_1(y)+ (-1)^l h_1(-y) \right) \nonumber
 \\
 = C \theta_1^l \left(x \bigm| \frac{\tau}{k}\right)\theta_1^l \left(y \bigm| \frac{\tau}{k}\right)
 \theta_1 \left(x+y \bigm| \frac{\tau}{k}\right) \theta_1 \left(x -y\bigm| \frac{\tau}{k}\right)
\end{align}
\end{theorem}
\begin{proof}
Following the notation of Liu~\cite{Liu-2007}, let
\begin{align*}
H(x) = \left( h_1(x)+ (-1)^l h_1(-x) \right)  \left( h_2(y)+ (-1)^l h_2(-y) \right)  \\
 -
 \left( h_2(x)+ (-1)^l h_2(-x) \right)  \left( h_1(y)+ (-1)^l h_1(-y) \right)
 \end{align*}
and let
\[
G(x) = \theta_1^l \left(x \bigm| \frac{\tau}{k}\right)
 \theta_1 \left(x+y \bigm| \frac{\tau}{k}\right) \theta_1 \left(x -y\bigm| \frac{\tau}{k}\right).
 \]
 Then it is easily checked, with the help on the assumptions on the functions $h_1$ and $h_2$,
  that
 \[
 H(x+\pi) = (-1)^l H(x) \ \text{and\ } H \lt x+\fr{\pi\tau}{k} \rt = (-1)^{l} q^{\fr{-(2+l)}{k}} e^{-2(2+l) ix} H(x).
 \]
 Moreover, by the basic properties in (\ref{theta-1}) and the formula (\ref{Transf}) we have
 \[
 G(x+\pi) = (-1)^l G(x) \ \text{and\ }  G \lt x+\fr{\pi\tau}{k} \rt = (-1)^{l} q^{\fr{-(2+l)}{k}} e^{-2(2+l) ix} G(x).
 \]
 Thus the ratio $\fr{H(x)}{G(x)}$ is an elliptic function with periods $\pi$ and $\fr{\pi\tau}{k}$. Suppose
 for the moment that $0<y<\pi$ and consider the fundamental period parallelogram with corners
 $0, \pi, \fr{\pi\tau}{k}, \pi+\fr{\pi\tau}{k}$. Clearly, in this parallelogram $G(x)$ has a zero of order $l$ at $x=0$
 and simple zeros at $x=y$ and $x=\pi-y$. Next, by the assumptions on $h_1(x)$ and
 $h_2(x)$ we can check that $H(x)$ has a zero of order at least $l$ at $x=0$ and zeros at $x=y$
 and $x=\pi-y$. So, $\fr{H(x)}{G(x)}$ has no poles in the period parallelogram and therefore
 $\fr{H(x)}{G(x)} = C(y)$ where $C(y)$ is a constant depending possibly only on $y$. That is,
 \begin{align*}
 \left( h_1(x)+ (-1)^l h_1(-x) \right)  \left( h_2(y)+ (-1)^l h_2(-y) \right) \\
 -
 \left( h_2(x)+ (-1)^l h_2(-x) \right)  \left( h_1(y)+ (-1)^l h_1(-y) \right) \\
 =
 C(y) \theta_1^l \left(x \bigm| \frac{\tau}{k}\right)
 \theta_1 \left(x+y \bigm| \frac{\tau}{k}\right) \theta_1 \left(x -y\bigm| \frac{\tau}{k}\right).
 \end{align*}
 Interchanging the roles of $x$ and $y$, we get
 \begin{align*}
 \left( h_1(x)+ (-1)^l h_1(-x) \right)  \left( h_2(y)+ (-1)^l h_2(-y) \right) \\
 -
 \left( h_2(x)+ (-1)^l h_2(-x) \right)  \left( h_1(y)+ (-1)^l h_1(-y) \right) \\
 =
 C(x) \theta_1^l \left(y \bigm| \frac{\tau}{k}\right)
 \theta_1 \left(x+y \bigm| \frac{\tau}{k}\right) \theta_1 \left(x -y\bigm| \frac{\tau}{k}\right).
 \end{align*}
 As the last two identities have equal left-hand-sides, we derive that
 \[
 \fr{C(y)}{\theta_1^l \left(y \bigm| \frac{\tau}{k}\right)} =
 \fr{C(x)}{\theta_1^l \left(x \bigm| \frac{\tau}{k}\right)},
 \]
 showing that $\fr{C(y)}{\theta_1^l \left(y \bigm| \frac{\tau}{k}\right)}$ is independent of $y$ as well.
 Thus we conclude that
 \[
 C(y) = C\te_1^l \left(y \bigm| \frac{\tau}{k}\right).
 \]
 In fact, the result extends to all complex numbers $x$ and $y$ by analytic continuation.
 This completes the proof.
\end{proof}
\noindent
Note that if $k=1$ and $l=1$, then Theorem~\ref{main-theorem} is \cite[Theorem 1]{Liu-2007}.
Note also that the constant $C$ is obtained by taking the $l$-th derivative derivative with respect to
$y$ in~(\ref{master}) as follows:
\begin{align}\label{Constant}
h_2^{(l)}(0) \left( h_1(x)+ (-1)^l h_1(-x) \right)-h_1^{(l)}(0)\left( h_2(x)+ (-1)^l h_2(-x) \right) \nonumber \\
= C \fr{l!}{2}\left( \te_1' \left(0\bigm| \frac{\tau}{k}\right) \right)^l
\te_1^{2+l}\left(x\bigm| \frac{\tau}{k}\right).
\end{align}
\section{Proof of Theorem \ref{main-result-1}}\label{sec-proof-main-result-1}
Let
\[
h_1(z) = \te_1^4\lt z\bigm|\fr{\tau'}{2} \rt \quad \text{and \quad}
h_2(z) = \te_1^2\lt z\bigm| \fr{\tau'}{4} \rt.
\]
 It can be checked with the help of the elementary
 properties in (\ref{theta-1}) that $h_1(z)$ and $h_2(z)$ both satisfy
the conditions of Theorem~\ref{main-theorem} for $k=l=2$. Then with the relation (\ref{Constant}) at hand
and some obvious simplifications,
the constant $C$ becomes
\[
C = 4 \lt \fr{ \te_1'\lt0\bigm| \fr{\tau'}{4} \rt}{\te_1'\lt0\bigm| \fr{\tau'}{2} \rt} \rt ^2.
\]
Now putting together in the formula (\ref{master}) we get
\begin{align*}
4 \te_1^4\lt x\bigm|\fr{\tau'}{2}\rt \te_1^2\lt y\bigm|\fr{\tau'}{4}\rt
- 4 \te_1^2\lt x\bigm|\fr{\tau'}{4}\rt \te_1^4\lt y\bigm|\fr{\tau'}{2}\rt \\
=
4 \lt \fr{ \te_1'\lt0\bigm| \fr{\tau'}{4} \rt}{\te_1'\lt0\bigm| \fr{\tau'}{2} \rt} \rt ^2
\te_1^2\lt x\bigm|\fr{\tau'}{2}\rt \te_1^2\lt y\bigm|\fr{\tau'}{2}\rt
\te_1\lt x+y\bigm|\fr{\tau'}{2}\rt \te_1\lt x-y\bigm|\fr{\tau'}{2}\rt.
\end{align*}
Letting $x=\fr{\pi}{2}$ in the previous identity yields after simplification
\begin{align}\label{help-1-1}
\te_1^2 \lt y\bigm|\fr{\tau'}{2} \rt \te_1^2 \lt y+\fr{\pi}{2}\bigm|\fr{\tau'}{2} \rt
=
\lt \fr{ \te_1'\lt0\bigm| \fr{\tau'}{2} \rt}{\te_1'\lt0\bigm| \fr{\tau'}{4} \rt} \rt ^2
\te_1^2 \lt \fr{\pi}{2}\bigm|\fr{\tau'}{2} \rt \te_1^2 \lt y\bigm|\fr{\tau'}{4} \rt \nonumber \\
-
\lt \fr{ \te_1'\lt0\bigm| \fr{\tau'}{2} \rt}{\te_1'\lt0\bigm| \fr{\tau'}{4} \rt} \rt ^2
\lt \fr{\te_1 \lt \frac{\pi}{2}\bigm| \fr{\tau'}{4} \rt} {\te_1 \lt \frac{\pi}{2}\bigm| \fr{\tau'}{2} \rt} \rt^2
\te_1^4 \lt y\bigm|\fr{\tau'}{2} \rt.
\end{align}
By formula (\ref{help}) we see that the left-hand-side of (\ref{help-1-1}) equals
\[
\fr{ \te_1^2 \lt 2y\bigm| \tau' \rt \te_1^4 \lt \fr{\pi}{4}\bigm| \fr{\tau'}{2}\rt}
{\te_1^2 \lt \fr{\pi}{2}\bigm| \tau' \rt}.
\]
Therefore, identity (\ref{help-1-1}) means
\[
\te_1^2 \lt 2y\bigm| \tau' \rt \te_1^4 \lt \fr{\pi}{4}\bigm| \fr{\tau'}{2}\rt =
\lt \fr{ \te_1'\lt0\bigm| \fr{\tau}{2} \rt}{\te_1'\lt0\bigm| \fr{\tau}{4} \rt} \rt ^2
\te_1^2 \lt \fr{\pi}{2}\bigm|\fr{\tau'}{2} \rt \te_1^2 \lt y\bigm|\fr{\tau'}{4} \rt
\te_1^2 \lt \fr{\pi}{2}\bigm| \tau' \rt  \]
\[
-
\lt \fr{ \te_1'\lt0\bigm| \fr{\tau'}{2} \rt}{\te_1'\lt0\bigm| \fr{\tau'}{4} \rt} \rt ^2
\lt \fr{\te_1 \lt \frac{\pi}{2}\bigm| \fr{\tau'}{4} \rt} {\te_1 \lt \frac{\pi}{2}\bigm| \fr{\tau'}{2} \rt} \rt^2
\te_1^4 \lt y\bigm|\fr{\tau'}{2} \rt \te_1^2 \lt \fr{\pi}{2}\bigm| \tau' \rt .
\]
Then multiplying both sides of the foregoing formula by
$\fr{\te_1^4 \lt \fr{\pi}{4}\bigm| \fr{\tau'}{2}\rt}{\te_1^4 \lt \fr{\pi}{4}\bigm| \fr{\tau'}{2}\rt}$
yields
\[
\te_1^2 \lt 2y\bigm| \tau' \rt \te_1^4 \lt \fr{\pi}{2}\bigm| \fr{\tau'}{2}\rt
\te_1^2 \lt \fr{\pi}{2}\bigm| \fr{\tau'}{4}\rt
\]
\[
= \lt \fr{ \te_1'\lt0\bigm| \fr{\tau'}{2} \rt}{\te_1'\lt0\bigm| \fr{\tau'}{4} \rt} \rt ^2
\fr{\te_1^6 \lt \frac{\pi}{2}\bigm| \fr{\tau'}{2} \rt \te_1^2 \lt \frac{\pi}{2}\bigm| \fr{\tau'}{4} \rt}
{\te_1^4 \lt \frac{\pi}{4}\bigm| \fr{\tau'}{2} \rt}
\te_1^2 \lt \frac{\pi}{2}\bigm| \tau' \rt \te_1^2 \lt y \bigm| \fr{\tau'}{4} \rt
\]
\[
-
\lt \fr{ \te_1'\lt0\bigm| \fr{\tau'}{2} \rt}{\te_1'\lt0\bigm| \fr{\tau'}{4} \rt} \rt ^2
\fr{\te_1^4 \lt \frac{\pi}{2}\bigm| \fr{\tau'}{4} \rt \te_1^2 \lt \frac{\pi}{2}\bigm| \fr{\tau'}{2} \rt}
{\te_1^4 \lt \frac{\pi}{4}\bigm| \fr{\tau'}{2} \rt}
\te_1^2 \lt \frac{\pi}{2}\bigm| \tau' \rt \te_1^4 \lt y \bigm| \fr{\tau'}{4} \rt
\]
\[
=
\lt \fr{ \te_1'\lt0\bigm| \fr{\tau'}{2} \rt}{\te_1'\lt0\bigm| \fr{\tau'}{4} \rt} \rt ^2
\fr{\te_1^2 \lt \frac{\pi}{2}\bigm| \fr{\tau'}{2} \rt \te_1^2 \lt \frac{\pi}{2}\bigm| \fr{\tau'}{4} \rt}
{\te_1^4 \lt \frac{\pi}{4}\bigm| \fr{\tau'}{2} \rt}
 \te_1^4 \lt \frac{\pi}{2}\bigm| \fr{\tau'}{2} \rt
\te_1^2 \lt \frac{\pi}{2}\bigm| \tau' \rt \te_1^2 \lt y \bigm| \fr{\tau'}{4} \rt
\]
\[
-
\lt \fr{ \te_1'\lt0\bigm| \fr{\tau'}{2} \rt}{\te_1'\lt0\bigm| \fr{\tau'}{4} \rt} \rt ^2
\fr{\te_1^2 \lt \frac{\pi}{2}\bigm| \fr{\tau'}{2} \rt \te_1^2 \lt \frac{\pi}{2}\bigm| \fr{\tau'}{4} \rt}
{\te_1^4 \lt \frac{\pi}{4}\bigm| \fr{\tau'}{2} \rt}
 \te_1^2 \lt \frac{\pi}{2}\bigm| \fr{\tau'}{4} \rt
  \te_1^2 \lt \frac{\pi}{2}\bigm| \tau' \rt  \te_1^4 \lt y \bigm| \fr{\tau'}{4} \rt.
\]
To complete the proof, it remains to show that
\begin{equation}\label{help-1-3}
\lt \fr{ \te_1'\lt0\bigm| \fr{\tau'}{2} \rt}{\te_1'\lt0\bigm| \fr{\tau'}{4} \rt} \rt ^2
\fr{\te_1^2 \lt \frac{\pi}{2}\bigm| \fr{\tau'}{2} \rt \te_1^2 \lt \frac{\pi}{2}\bigm| \fr{\tau'}{4} \rt}
{\te_1^4 \lt \frac{\pi}{4}\bigm| \fr{\tau'}{2} \rt}
=
\frac{1}{4} \lt \fr{\Pi_q}{\Pi_{q^4}} \rt ^2.
\end{equation}
By virtue of the relations (\ref{ImTransf}) and (\ref{key-derivative}) we get
\[
\begin{split}
\fr{\te_1^2\lt \fr{\pi}{2}\bigm| \fr{\tau'}{2}\rt}{\te_1^2\lt \fr{\pi}{4}\bigm| \fr{\tau'}{2}\rt} &=
q^{-\fr{1}{4}} \fr{(q^2;q^4)_{\infty}^4}{(q;q^2)_{\infty}^2}, \\
\fr{\te_1^2\lt \fr{\pi}{2}\bigm| \fr{\tau'}{4}\rt}{\te_1^2\lt \fr{\pi}{4}\bigm| \fr{\tau'}{2}\rt} &=
2 q^{-\fr{1}{4}} \fr{(q^4;q^8)_{\infty}^4 (q^8;q^8)_{\infty}^2}{(q;q^2)_{\infty}^2 (q^4;q^4)_{\infty}^2}, \\
\lt \fr{ \te_1'\lt0\bigm| \fr{\tau'}{2} \rt}{ \te_1'\lt0\bigm| \fr{\tau'}{4} \rt} \rt^2 &=
\fr{1}{8} q^{-1} \fr{(q^4;q^4)_{\infty}^6}{(q^8;q^8)_{\infty}^6}.
\end{split}
\]
Then by the previous three ratios we see that identity (\ref{help-1-3}) is equivalent to
\[
\fr{1}{8} q^{-1} \fr{(q^4;q^4)_{\infty}^6}{(q^8;q^8)_{\infty}^6}  \cdot
q^{-\fr{1}{4}} \fr{(q^2;q^4)_{\infty}^4}{(q;q^2)_{\infty}^2}  \cdot
2 q^{-\fr{1}{4}} \fr{(q^4;q^8)_{\infty}^4 (q^8;q^8)_{\infty}^2}{(q;q^2)_{\infty}^2 (q^4;q^4)_{\infty}^2}
\]
\[ =
\fr{1}{4} q^{-\fr{3}{2}}
\fr{(q^2;q^2)_{\infty}^4 (q^4;q^8)_{\infty}^4}{(q;q^2)_{\infty}^4 (q^8;q^8)_{\infty}^4},
\]
or equivalently,
\[
\fr{1}{4} q^{-\fr{3}{2}} (q^4;q^4)_{\infty}^4 (q^2;q^4)_{\infty}^4 = \fr{1}{4} q^{-\fr{3}{2}} (q^2;q^2)_{\infty}^4,
\]
which is obviously true. This completes the proof.
\section{Proof of Theorem \ref{main-result-2}}\label{sec-proof-main-result-2}
We can check that the functions
$h_1(z) = \te_1^3\lt z\bigm| \fr{\tau'}{3} \rt$ and
$h_2(z) = \te_1 ( 3z\bigm| \tau' )$ satisfy the conditions of Theorem~\ref{main-theorem} for
$k=3$ and $l=1$. Now letting in Theorem~\ref{main-theorem} $x=\fr{\pi}{2}$ and $y=z$,
 and using the formula~(\ref{key-derivative}) give
\begin{align}\label{help-2-0}
4 \te_1^3 \lt \fr{\pi}{2}\bigm| \fr{\tau'}{3} \rt \te_1 (3z\mid \tau') -
4 \te_1^3 \lt z\bigm| \fr{\tau'}{3} \rt \te_1\lt \fr{3\pi}{2}\bigm| \tau' \rt  \nonumber \\
=
12 \fr{\te_1'(0\mid\tau')} {\te_1' \lt 0\bigm| \fr{\tau'}{3} \rt} \te_1 \lt z\bigm| \fr{\tau'}{3} \rt
\te_1 \lt \fr{\pi}{2}+z\bigm| \fr{\tau'}{3} \rt \te_1 \lt \fr{\pi}{2}-z\bigm| \fr{\tau'}{3} \rt
\te_1 \lt \fr{\pi}{2}\bigm| \fr{\tau'}{3} \rt.
\end{align}
Then multiplying both sides of the foregoing identity by $\te_1 \lt \fr{\pi}{2}\bigm| \fr{\tau'}{9} \rt$, using the basic
properties~(\ref{theta-1}), and simplifying we conclude that the last identity means
\begin{align} \label{help-2-1}
\te_1 (3z\mid \tau') \te_1^3 \lt \fr{\pi}{2}\bigm| \fr{\tau'}{3} \rt \te_1 \lt \fr{\pi}{2}\bigm| \fr{\tau'}{9} \rt
+
\te_1^3 \lt z\bigm| \fr{\tau'}{3} \rt \te_1\lt \fr{\pi}{2}\bigm| \tau' \rt \te_1 \lt \fr{\pi}{2}\bigm| \fr{\tau'}{9} \rt
\nonumber \\
=
3 \fr{\te_1'(0\mid\tau')} {\te_1' \lt 0\bigm| \fr{\tau'}{3} \rt} \te_1 \lt z\bigm| \fr{\tau'}{3} \rt
\te_1^2 \lt z+ \fr{\pi}{2}\bigm| \fr{\tau'}{3} \rt \te_1 \lt \fr{\pi}{2}\bigm| \fr{\tau'}{3} \rt
\te_1 \lt \fr{\pi}{2}\bigm| \fr{\tau'}{9} \rt.
\end{align}
On the other hand,
 in Theorem~\ref{main-theorem}, let this time letting,
 \[
 h_1(z) = \te_1^3\lt z\bigm| \fr{\tau'}{3} \rt \quad \text{and\quad}
h_2(z) = \te_1 \lt z\bigm| \fr{\tau'}{9} \rt,
\]
with $k=3$ and $l=1$
 and using the formula~(\ref{key-derivative}) we get
\begin{align*}
4 \te_1^3\lt x\bigm| \fr{\tau'}{3} \rt \te_1\lt y\bigm| \fr{\tau'}{9} \rt -
4 \te_1\lt x\bigm| \fr{\tau'}{9} \rt \te_1^3 \lt y\bigm| \fr{\tau'}{3} \rt \\
=
4 \fr{ \te_1'\lt 0\bigm| \fr{\tau'}{9} \rt}{ \te_1'\lt 0\bigm| \fr{\tau'}{3} \rt} \te_1\lt x\bigm| \fr{\tau'}{3} \rt
\te_1\lt y\bigm| \fr{\tau'}{3} \rt \te_1\lt x+y\bigm| \fr{\tau'}{3} \rt \te_1\lt x-y\bigm| \fr{\tau'}{3} \rt.
\end{align*}
Now make the substitution $x=\fr{\pi}{2}$ and $z=y$ in the previous identity and multiply both sides by
$A_q:=\fr{1}{3} \fr{\Pi_q}{\Pi_{q^9}}$ to obtain
\begin{align} \label{help-2-2}
\te_1\lt z\bigm| \fr{\tau'}{9} \rt \te_1^3\lt \fr{\pi}{2}\bigm| \fr{\tau'}{3}  \rt A_q
-
\te_1^3 \lt z\bigm| \fr{\tau'}{3} \rt \te_1\lt \fr{\pi}{2}\bigm| \fr{\tau'}{9}  \rt A_q  \nonumber \\
=
\fr{ \te_1'\lt 0\bigm| \fr{\tau'}{9} \rt}{ \te_1'\lt 0\bigm| \fr{\tau'}{3} \rt}
\te_1\lt z\bigm| \fr{\tau'}{3} \rt \te_1^2 \lt z+\fr{\pi}{2} \bigm| \fr{\tau'}{3} \rt
\te_1\lt \fr{\pi}{2}\bigm| \fr{\tau'}{3} \rt A_q
\end{align}
Clearly, the desired result holds if and only if the right-hand sides of (\ref{help-2-1}) and (\ref{help-2-2})
are equal. Therefore, we will be done if we derive that
\begin{equation*}
3 \te_1'(0\mid \tau') \te_1\lt \fr{\pi}{2}\bigm| \fr{\tau'}{9} \rt =
\fr{1}{3} \te_1' \lt 0\bigm| \fr{\tau'}{9} \rt \te_1\lt \fr{\pi}{2}\bigm| \tau' \rt \fr{\Pi_q}{\Pi_{q^9}},
\end{equation*}
or equivalently,
\begin{equation}\label{help-2-3}
9 \fr{ \te_1'(0\mid \tau')}{\te_1' \lt 0\bigm| \fr{\tau'}{9} \rt}
\fr{ \te_1\lt \fr{\pi}{2}\bigm| \fr{\tau'}{9} \rt }
{ \te_1\lt \fr{\pi}{2}\bigm| \tau' \rt}
=
\fr{\Pi_q}{\Pi_{q^9}}.
\end{equation}
By an appeal to the relations (\ref{ImTransf}) and (\ref{key-derivative}) we find
\[
\begin{split}
\fr{\te_1\lt \fr{\pi}{2}\bigm| \fr{\tau'}{9}\rt}{\te_1^2\lt \fr{\pi}{2}\bigm| \tau' \rt} &=
\sqrt{9} \fr{(q^9;q^{18})_{\infty}^2 (q^{18};q^{18})_{\infty}}{(q;q^{2})_{\infty}^2 (q^2;q^2)}, \\
 \fr{ \te_1'\lt0\bigm| \tau' \rt}{ \te_1'\lt0\bigm| \fr{\tau'}{9} \rt}  &=
\fr{1}{9 \sqrt{9}} q^{-2} \fr{(q^2;q^2)_{\infty}^3}{(q^{18};q^{18})_{\infty}^3},
\end{split}
\]
which yields
\[
9 \fr{ \te_1'(0\mid \tau')}{\te_1' \lt 0\bigm| \fr{\tau'}{9} \rt}
\fr{ \te_1\lt \fr{\pi}{2}\bigm| \fr{\tau'}{9} \rt }
{ \te_1\lt \fr{\pi}{2}\bigm| \tau' \rt}
=
q^{-2} \fr{ (q^2;q^2)_{\infty}^2 (q^9;q^{18})_{\infty}}{ (q^{18};q^{18})_{\infty}^2 (q;q^2)_{\infty}^2}
= \fr{\Pi_q}{\Pi_{q^9}},
\]
as desired in formula (\ref{help-2-3}).
\section{Proof of Theorem \ref{main-result-3}} \label{sec-proof-main-result-3}
By the formula (\ref{help-2-0}) we have
\begin{align*}
 \te_1 (3z\mid \tau') \te_1^3 \lt \fr{\pi}{2}\bigm| \fr{\tau'}{3} \rt +
\te_1^3 \lt z\bigm| \fr{\tau'}{3} \rt \te_1\lt \fr{\pi}{2}\bigm| \tau' \rt  \\
=
3 \fr{\te_1'(0\mid\tau')} {\te_1' \lt 0\bigm| \fr{\tau'}{3} \rt} \te_1 \lt z\bigm| \fr{\tau'}{3} \rt
\te_1^2 \lt z+ \fr{\pi}{2} \bigm| \fr{\tau'}{3} \rt
\te_1 \lt \fr{\pi}{2}\bigm| \fr{\tau'}{3} \rt.
\end{align*}
Therefore, to prove the theorem we need only show that
\begin{equation}\label{help-3-1}
3 \fr{\te_1'(0\mid\tau')} {\te_1' \lt 0\bigm| \fr{\tau'}{3} \rt}
\fr{\te_1 \lt \fr{\pi}{2}\bigm| \fr{\tau'}{3}  \rt}{\te_1 \lt \fr{\pi}{2}\bigm| \tau' \rt } = \fr{\prod_q}{\prod_{q^3}}.
\end{equation}
By  the idenitites (\ref{ImTransf}) and (\ref{key-derivative}) we have
\[
\begin{split}
\fr{\te_1\lt \fr{\pi}{2}\bigm| \fr{\tau'}{3}\rt}{\te_1^2\lt \fr{\pi}{2}\bigm| \tau' \rt} &=
\sqrt{3} \fr{(q^3;q^{6})_{\infty}^2 (q^{6};q^{6})_{\infty}}{(q;q^{2})_{\infty}^2 (q^2;q^2)}, \\
 \fr{ \te_1'\lt0\bigm| \tau' \rt}{ \te_1'\lt0\bigm| \fr{\tau'}{3} \rt}  &=
\fr{1}{3 \sqrt{3}} q^{\fr{-1}{2}} \fr{(q^2;q^2)_{\infty}^3}{(q^{6};q^{6})_{\infty}^3},
\end{split}
\]
from which we get
\[
3 \fr{ \te_1'(0\mid \tau')}{\te_1' \lt 0\bigm| \fr{\tau'}{3} \rt}
\fr{ \te_1\lt \fr{\pi}{2}\bigm| \fr{\tau'}{3} \rt }
{ \te_1\lt \fr{\pi}{2}\bigm| \tau' \rt}
=
q^{\fr{-1}{2}} \fr{ (q^2;q^2)_{\infty}^2 (q^3;q^{6})_{\infty}^2}{ (q;q^2)_{\infty}^2 (q^{6};q^{6})_{\infty}^2}
= \fr{\Pi_q}{\Pi_{q^3}},
\]
as desired in formula (\ref{help-3-1}). This completes the proof.

\vspace{0.5cm}

\noindent
%
\noindent{\bf Acknowledgment.} The author is grateful to the referee
 for valuable comments and interesting suggestions.

%
\end{document}